\documentclass[amstex,16pt, amssymb]{article}
\usepackage{mathtext}
\usepackage[cp1251]{inputenc}
\usepackage[T2A]{fontenc}
\usepackage[dvips]{graphicx}
\usepackage{amsmath}
\usepackage{amssymb}
\usepackage{amsxtra}
\usepackage{latexsym}
\usepackage{ifthen}
\usepackage{amsthm}

\def\XXint#1#2#3{{\setbox0=\hbox{$#1{#2#3}{\int}$}
     \vcenter{\hbox{$#2#3$}}\kern-.5\wd0}}

\theoremstyle{plain}
\newtheorem{theorem}{Theorem}[section]
\newtheorem{lemma}{Lemma}[section]
\newtheorem{proposition}{Proposition}[section]
\newtheorem{corollary}{Corollary}[section]

\theoremstyle{definition}
\newtheorem{example}{Example}[section]

\renewcommand{\abstract}{\textbf{Abstract. }\medskip}

\numberwithin{equation}{section}

\numberwithin{equation}{section}

\begin{document}

\title{Asymptotic behavior of solutions of the nonlinear Beltrami equation with the Jacobian}

\author{Igor Petkov, Ruslan Salimov, Mariia Stefanchuk}

\date{}

\maketitle

\begin{abstract}
We investigate the asymptotic behavior at infinity of regular homeomorphic solutions of the nonlinear Beltrami equation with the Jacobian 
on the right-hand side. The sharpness of the above bounds is illustrated by several examples. 
\end{abstract}

\bigskip

{\bf MSC 2020.} {Primary: 28A75, 30A10; Secondary: 26B15, 30C80}

\bigskip

{\bf Key words.} {Beltrami equations, nonlinear Beltrami equations, So\-bolev class, asymptotic behavior at infinity, regular  homeomorphism, angular dilatation, isoperimetric inequalities.}


\section{Introduction}


Nowadays various kinds of nonlinear counterparts of the classical Beltra\-mi equation attarct attention of many mathematicians. There are se\-veral tools for studying the main features of such equations and asymptotic behaviers of their solutions. The so-called directional dilatations and isoperimetric inequality allow us to establish some crutial properties of the regular solutions; see, e.g. \cite{AC}--\cite{BGMR13}.

Let $\mathbb{C}$ be the complex plane. In the complex notation $f=u+iv$ and $z=x+iy$, the {\it Beltrami equation} in a domain $G\subset\mathbb{C}$ has the form
\begin{equation}\label{eqB}
f_{\overline{z}}\,=\,\mu(z)f_{z},
\end{equation}
where $\mu\colon G\to \mathbb{C}$ is a measu\-rable function and
$$f_{\overline{z}}=\frac{1}{2}(f_{x}+if_{y})\qquad {\rm and} \qquad f_{z}=\frac{1}{2}(f_{x}-if_{y})$$
are formal derivatives of $f$ in $\overline{z}$ and $z$, while $f_{x}$ and $f_{y}$  are partial derivatives of $f$ in the variables $x$ and $y$, respectively.

Various existence theorems for solutions of the Sobolev class  have been recently established applying the modulus approach for a quite wide class of linear and quasilinear degenerate Beltrami equations; see, e.g. \cite{RSY}, \cite{GRSY12}--\cite{SY05}, \cite{Sev11}, \cite{AIM09}.

\medskip

Let $z_0\in \mathbb{C}$ and $\mathcal{K}_{z_0} \colon G \to\mathbb{C}$ be a measurable function. We consider the following equation

\begin{equation}\label{eqcomplexJ}
f_{\overline{z}}\,-\frac{z-z_0}{\overline{z-z_0}}\,f_{z}=\mathcal{K}_{z_0}(z) |J_f(z)|^{1/2} \,,
\end{equation}
where  $J_{f}(z)=|f_z|^{2}-|f_{\bar{z}}|^{2}$ is a Jacobian of $f$. 

Under $\mathcal{K}_{z_0} (z)\equiv0$ equation (\ref{eqcomplexJ}) reduces to the standard linear Beltrami equation \eqref{eqB} with the complex coefficient
$\mu(z)\,=\,\frac{z-z_0}{\overline{z-z_0}}.$
In other cases,  the equation (\ref{eqcomplexJ}) provides a partial case of the general nonlinear system of equations (7.33) given in \cite[Sect.~7.7]{AIM09}.

Applying  the formal derivatives
\begin{equation}\label{eqzz}
rf_{r}\,=\,(z-z_0)f_{z}+\left(\overline{z-z_0}\right)f_{\overline{z}}\,, \quad f_{\theta}\,=\,i((z-z_0)f_{z}-\left(\overline{z-z_0}\right)f_{\overline{z}})\,,
\end{equation}
see, e.g. \cite[(21.51)]{AIM09}, one can rewrite the equation \eqref{eqcomplexJ}   in the polar coordinates $(r,\theta)$ $(z=z_0+re^{i\theta})$:
\begin{equation}\label{eqJac}f_{\theta}=\sigma_{z_0}|J_{f}|^{1/2}\end{equation}
with \begin{equation}\label{sigmaK}\sigma_{z_0}=\sigma_{z_0}(z)=-i\mathcal{K}_{z_0} (z)(\overline{z-z_0})
\end{equation}
and 
\begin{equation}\label{Jacob}
J_{f}=J_{f}(z_0+re^{i\theta})=\frac{1}{r} \, {\rm Im}\, \left(\overline{f_r}\, f_\theta \right),
\end{equation}
where $f_{\theta}$ and $f_r$ are the partial derivatives of $f$ by $\theta$ and $r$, respectively, see, e.g. \cite[(21.52)]{AIM09}. 

Next, in the case $z_0=0$ we put $\mathcal{K}_{z_0}(z)=\mathcal{K}(z)$ and $\sigma_{z_0}(z)=\sigma(z)$.

\begin{example}
Let $A$, $B$, $C$ are complex numbers and $|A|\neq|B|$. Note that the linear mapping 
$$f(z)=A\overline{z}+Bz+C$$
is the solution of the equation \eqref{eqcomplexJ} with 
$$\mathcal{K}_{z_0}(z)=\frac{A(\overline{z-z_0})-B(z-z_0)}{|\Delta|^{\frac{1}{2}}(\overline{z-z_0})},$$
where $\Delta=|B|^{2}-|A|^{2}\neq0$.
\end{example}

\begin{example}
Let us consider the following area preserving quasiconfor\-mal mapping
$$f(z)=ze^{2i\ln|z|}$$
in $\mathbb{B}$. We see that 
$$f_{z}=(1+i)e^{2i\ln|z|}, \quad f_{\overline{z}}=\frac{iz}{\overline{z}}\,e^{2i\ln|z|},$$
and therefore the straight forward computation shows that
$$J_f=|f_z|^{2}-|f_{\bar{z}}|^{2}=|1+i|^{2}-1=1$$
for $z\in\mathbb{B}$.

Thus, it is obvious that the mapping $f$ is a solution of the equation 
$$f_{\overline{z}}\,-\frac{z}{\overline{z}}\,f_{z}=\mathcal{K}(z) |J_f(z)|^{1/2}$$
with $\mathcal{K}(z)=-\frac{z}{\overline{z}}\,e^{2i\ln|z|}$.
\end{example}

Note that the equation \eqref{eqcomplexJ} can be written in the form of a system of two real partial differential equations
$$
\begin{cases}(y-y_{0})u_{x}-(x-x_{0})u_{y}=k_{1}|J_f|^{1/2}, \\ (y-y_{0})v_{x}-(x-x_{0})v_{y}=k_{2}|J_f|^{1/2},
\end{cases}
$$
where $k_{1}=-{\rm Im}((\overline{z-z_{0}})\,\mathcal{K}_{z_0}(z))$, $k_{2}={\rm Re}((\overline{z-z_{0}})\,\mathcal{K}_{z_0}(z))$, $z=x+iy$ and $z_{0}=x_{0}+iy_{0}$. 

\medskip

The nonlinear equation (\ref{eqcomplexJ}) provides  partial cases of the nonlinear system of two real partial differential equations; see (1) in \cite{LSh57}, \cite{Lavr47}, cf. \cite{Lavr62}. Note that various nonlinear systems of partial differential equations studied in a quite large specter of aspects can be found in \cite{GSS}--\cite{SalSt4}, \cite{AIM09}--\cite{CGPSS18}.

\section{Auxiliary Results}

A mapping $f \colon G \to \mathbb C$ is called
\emph{regular at a point} $z_0\in G,$ if $f$ has the total differential at this point and its Jacobian
$J_f=|f_z|^{2}-|f_{\bar{z}}|^{2}$ does not vanish, cf. \cite[I.~1.6]{LV73}. A homeomorphism $f$ of Sobolev class $W_{\rm{loc}}^{1, 1}$ is called \emph{regular}, if $J_{f}>0$ a.e. (almost everywhere). 
By a {\it regular homeomorphic solution of equation} (\ref{eqcomplexJ}) we call a regular homeomorphism $f \colon G \to \mathbb C,$ which satisfies (\ref{eqcomplexJ}) a.e. in $G.$

Later on we use the following notations
\begin{equation*}
B(z_0,r)=\{z\in\mathbb C: |z-z_0|< r\} \quad {\rm and} \quad \gamma(z_0,r)=\{z\in\mathbb C:
|z-z_0|=r\}.
\end{equation*}
Given a set $E\in\mathbb{C}$, $|E|$ denotes the two dimensional Lebesgue measure of a set $E$. We denote by $S_f(z_0, r)=|f(B(z_0,r))|.$ 

The mapping  $f: G \to \mathbb C $ has
the {\it N-property} (by Luzin), if the condition $|E|=0$ implies that $|f(E)|=0$.

Let $G$ be a domain in $\mathbb{C}$. Let $f:G\to  \mathbb{C}$ be a regular homeomorphism of the Sobolev class $W_{\rm loc }^{1,1}$. Given a point 
$z_0\in G$, the  {\it angular dilatation}
of $f$ with respect to $z_0$ is the function

$$D_{f}(z,z_0) =  \frac{|f_{\theta}(z_0+ re^{i\theta})|^2}{r^2 J_f(z_0+re^{i\theta})}\,,$$
where $z =z_0+ re^{i\theta}$ and  $J_f$ is the Jacobian of $f$, see the equation (3.10) in \cite{BGMR13}. 

For $D_{f}(z,z_0)$ denote

\begin{equation}\label{defd_p}
d_{f}(z_0,r)= \frac{1}{2\pi r} \int\limits_{\gamma(z_0,r)}\, D_{f}(z,z_0) \, |dz|\,.
\end{equation}

\begin{proposition}\label{pr1}
Let $f\colon  \mathbb C\to \mathbb
C$ be a regular homeomorphism of the Sobolev class $W^{1, 1}_{\rm
loc}$ that possesses the $N$-property, $z_0\in \mathbb{C}$. Then
\begin{equation}\label{eqlem1}
S_{f}^\prime(z_0,r)\geqslant 2 \,\frac{S_{f}(z_0,r)}{r \,  d_{f}(z_0,r)}
\end{equation}
for a.a. (almost all) $r\in (0,+\infty)$.
\end{proposition}

\begin{proof} Denote by $L_{f}(z_0, r)$ the length of curve $f(z_0+re^{i\theta}),$ $0\leqslant\theta\leqslant 2\pi.$ For a.a. $r\in (0,+\infty),$
\begin{equation*}
L_{f}(z_0, r)=\int\limits_{0}^{2\pi}|f_{\theta}(z_0+re^{i\theta})|\,d\theta=\int\limits_{0}^{2\pi}
 D_f^\frac{1}{2}(z_0+re^{i\theta},z_0)
J_f^\frac{1}{2}(z_0+re^{i\theta})\, r\,d\theta\,,
\end{equation*}
and by H\"{o}lder's inequality,
\begin{equation}\label{eqlem11}
L_{f}^{2}(z_0,r)\,\leqslant\,
\int\limits_{0}^{2\pi}D_f(z_0+re^{i\theta},z_0)\, r\,d\theta
\int\limits_{0}^{2\pi}J_f(z_0+re^{i\theta})\, r\,d\theta\,.
\end{equation}

Due to Lusin's $N$-property and the Fubini theorem,
\begin{equation*}
S_{f}(z_0, r)=\int\limits_{B(z_0,r)}J_{f}(z)\,dxdy=\int\limits_{0}^{r}\int\limits_{0}^{2\pi}J_{f}(z_0+\rho e^{i\theta})\rho\,d\theta\,  d\rho\,,
\end{equation*}
hence, for a.a. $r\in (0,+\infty)$
\begin{equation*}
S_{f}^\prime(z_0, r)\,=\,\int\limits_{0}^{2\pi}
J_{f}(z_0+re^{i\theta}) \, r\,d\theta\,.
\end{equation*}
Estimating the last integral by (\ref{eqlem11}) and using (\ref{defd_p}), one obtains
\begin{equation}\label{eqlem12}
S_{f}^\prime(z_0,r)\,\geqslant\,\frac{L_{f}^2(z_0,r)}{2\pi r\,d_f(z_0,r)}
\end{equation}
for a.a. $r\in (0,+\infty)$. Combining (\ref{eqlem12}) with the planar isoperimetric inequality
$$L_{f}^2(z_0,r)\geqslant 4\pi S_{f}(z_0,r)$$
implies the desired relation (\ref{eqlem1}).
\end{proof}

\begin{proposition}\label{pr2}
Let $f\colon  \mathbb C\to \mathbb
C$ be a regular homeomorphic solution of the equation (\ref{eqcomplexJ}) which belongs to  Sobolev class $W^{1, 2}_{\rm
loc}$. Then
\begin{equation}\label{diffeq}
S_{f}^\prime(z_0,r)\geqslant 2 \,\frac{S_{f}(z_0,r)}{r \,  \kappa(z_0,r)}
\end{equation}
for a.a.  $r\in (0,+\infty)$ and here $$\kappa(z_0,r)=\frac{1}{2\pi r} \, \int\limits_{\gamma(z_0, r)} | \mathcal{K}_{z_0} (z)|^2 \, |dz|.$$
\end{proposition}

\begin{proof} We note that, according to Corollary B in \cite{MM}, the homeomor\-phism $W_{\rm loc }^{1,2}$
possesses the $N$-property. Since $f$ is a regular homeomor\-phic solution of equation  (\ref{eqcomplexJ}), we get

$$ D_f(z,z_0)= \frac{|f_{\theta}(z_0+ re^{i\theta})|^2}{r^2 J_f(z_0+re^{i\theta})}=\frac{|\sigma_{z_0}(z_0+ re^{i\theta})|^2}{r^2}\,, $$
where $z =z_0+ re^{i\theta}$.

Next,  due to the   relation (\ref{sigmaK}),  we obtain   $$D_f(z,z_0)= |\mathcal{K}_{z_0} (z)|^2\,$$
and
$$ d_{f}(z_0,r)= \frac{1}{2\pi r} \int\limits_{\gamma(z_0,r)}\, D_{f}(z,z_0) \, |dz|= \frac{1}{2\pi r} \int\limits_{\gamma(z_0,r)}\, |\mathcal{K}_{z_0} (z)|^2 \, |dz|\,. $$
Thus,  applying Proposition \ref{pr1}, we obtain (\ref{diffeq}).

\end{proof}

\begin{lemma}\label{mesrR} Let $f\colon  \mathbb C\to \mathbb
C$ be a regular homeomorphic solution of the equation (\ref{eqcomplexJ}) which belongs to  Sobolev class $W^{1, 2}_{\rm
loc}$. Then 

\begin{equation}
S_{f}(z_0,r_0) \leqslant S_{f}(z_0,R) \, \exp\left(-  2 \int\limits_{r_0}^R \,\frac{dr}{r \,  \kappa(z_0,r)}\right)\,
\end{equation} 
for all $R>r_{0}>0$.

\end{lemma}

\begin{proof}
Let $0<r_{0}<R<\infty$. By Proposition \ref{pr2}, for a.a. $r\in (0,+\infty)$
\begin{equation*}
\frac{S_{f}^\prime(z_0,r)}{S_{f}(z_0,r)}\,dr\geqslant
2\, \frac{dr}{r \,  \kappa(z_0,r)}
\end{equation*}
and integrating over the segment $[r_{0},R],$ we obtain
\begin{equation*}
\int\limits_{r_{0}}^{R}\frac{S_{f}^\prime(z_0,r)}{S_{f}(z_0,r)}\,dr\geqslant
2\int\limits_{r_{0}}^{R}\frac{dr}{r \,  \kappa(z_0,r)}.
\end{equation*}
Hence,
\begin{equation*}
\int\limits_{r_{0}}^{R}\left(\ln S_{f}(z_0,r)\right)^\prime dr\geqslant
2\int\limits_{r_{0}}^{R}\frac{dr}{r \,  \kappa(z_0,r)}.
\end{equation*}
Note that the function $\psi(r)=\ln S_{f}(z_0,r)$ is nondecreasing on $(0,+\infty)$, and
\begin{equation*}
\int\limits_{r_{0}}^{R}\left(\ln S_{f}(z_0,r)\right)^\prime dr=
\int\limits_{r_{0}}^{R}\psi^\prime(r) dr\leqslant \psi(R) - \psi(r_{0})=\ln\frac{S_{f}(z_0,R)}{S_{f}(z_0,r_{0})},
\end{equation*}
see Theorem IV.7.4 in \cite{Saks64}.
Combining the last two inequalities, we have
\begin{equation*}
\ln\frac{S_{f}(z_0,R)}{S_{f}(z_0,r_{0})}\geqslant2\int\limits_{r_{0}}^{R}\frac{dr}{r \,  \kappa(z_0,r)}.
\end{equation*}
Thus, we obtain that 
\begin{equation*}
S_{f}(z_0,r_0) \leqslant S_{f}(z_0,R) \, \exp\left(-  2 \int\limits_{r_0}^R \,\frac{dr}{r \,  \kappa(z_0,r)}\right).
\end{equation*} 
The lemma is proved.
\end{proof}

\section{Asymptotic behavior at infinity of regular ho\-meomorphic solutions}

Here we study an asymptotic behavior at infinity of regular homeomor\-phic solutions of the equation (\ref{eqcomplexJ}).

\begin{theorem}\label{th1} Let $f\colon  \mathbb C\to \mathbb
C$ be a regular homeomorphic solution of the equation (\ref{eqcomplexJ}) which belongs to  Sobolev class $W^{1, 2}_{\rm
loc}$, $r_0>0$. Then 
\begin{equation}\label{eq11}
\liminf\limits_{R\to \infty}M_{f}(z_0,R)\exp\left(-  \int\limits_{r_0}^R \,\frac{dr}{r \,  \kappa(z_0,r)}\right)\geqslant m_{f}(z_0,r_0)>0,
\end{equation} 
where $$M_{f}(z_0,R)=\max\limits_{|z-z_0|=R}|f(z)-f(z_0)|$$ and $$m_{f}(z_0,r_0)=\min\limits_{|z-z_0|=r_0}|f(z)-f(z_0)|.$$
\end{theorem}
\begin{proof}
By Lemma \ref{mesrR}, we have 
$$S_{f}(z_0,r_0) \leqslant S_{f}(z_0,R) \, \exp\left(-  2 \int\limits_{r_0}^R \,\frac{dr}{r \,  \kappa(z_0,r)}\right)$$
for all $R>r_{0}>0$. Since $f$ is homeomorphism, we obtain  
$$\pi\, m_{f}^{2}(z_0,r_0)\leqslant S_{f}(z_0,r_0) \leqslant S_{f}(z_0,R) \, \exp\left(-  2 \int\limits_{r_0}^R \,\frac{dr}{r \,  \kappa(z_0,r)}\right)
\leqslant$$
$$\leqslant \pi\, M_{f}^{2}(z_0,R) \, \exp\left(-  2 \int\limits_{r_0}^R \,\frac{dr}{r \,  \kappa(z_0,r)}\right).$$
Hence, 
$$m_{f}(z_0,r_0)\leqslant M_{f}(z_0,R) \, \exp\left(-  \int\limits_{r_0}^R \,\frac{dr}{r \,  \kappa(z_0,r)}\right).$$
Finally, passing to the lower limit as $R\to \infty$ in the last inequality, we obtain the relation \eqref{eq11}.
\end{proof}

\begin{corollary}\label{cor1} Let $f\colon  \mathbb C\to \mathbb
C$ be a regular homeomorphic solution of the equation (\ref{eqcomplexJ}) which belongs to  Sobolev class $W^{1, 2}_{\rm
loc}$ and $\alpha>0$. If $\kappa(z_0,r)\leqslant\alpha$ for a.a. $r\geqslant 1$, then 
$$
\liminf\limits_{R\to \infty}\frac{M_{f}(z_0,R)}{R^{1/\alpha}}\geqslant m_{f}(z_0,1)>0,
$$
where $m_{f}(z_0,1)=\min\limits_{|z-z_0|=1}|f(z)-f(z_0)|.$
\end{corollary}

\begin{example}
Assume that $\alpha>0$. Consider the equation
\begin{equation}\label{1111}
f_{\overline{z}}\,-\frac{z}{\overline{z}}\,f_{z}=-\alpha^{\frac{1}{2}}\frac{z}{\overline{z}}\, |J_f(z)|^{\frac{1}{2}}
\end{equation}
in the complex plane $\mathbb{C}.$ In the polar coordinates system, this equation takes the form
$$f_{\theta}=i\alpha^{\frac{1}{2}}re^{i\theta}|J_{f}|^{\frac{1}{2}}.$$
It's easy to check that the mapping $$f=\begin{cases}|z|^{\frac{1}{\alpha}-1}z,\,\,z\neq0,\\
0,\,\,z=0\end{cases}$$
is regular homeomorphism and  belongs to Sobolev class $W_{\rm loc}^{1,2}(\mathbb{C})$.
We show that $f$ is a solution of the equation \eqref{1111}. We write this mapping in the polar coordinates $f(re^{i\theta})=r^{\frac{1}{\alpha}}e^{i\theta}$.
The partial derivatives of $f$ by $r$ and $\theta$ are
\begin{equation*}
f_{r}=\frac{1}{\alpha}\,r^{\frac{1-\alpha}{\alpha}}e^{i\theta}, \quad f_{\theta}=ir^{\frac{1}{\alpha}}e^{i\theta}
\end{equation*}
and by \eqref{Jacob} we have $$J_f(re^{i\theta})=\frac{1}{\alpha}\,r^{\frac{2(1-\alpha)}{\alpha}}>0.$$
Next we find $$\sigma=\frac{f_{\theta}}{J_f^{\frac{1}{2}}}=i\alpha^{\frac{1}{2}}re^{i\theta}.$$
Consequently, $\sigma=i\alpha^{\frac{1}{2}}z$ and by the relation \eqref{sigmaK} we obtain $$\mathcal{K}(z)=-\frac{\sigma(z)}{i\,\overline{z}}=-\alpha^{\frac{1}{2}}\frac{z}{\overline{z}}$$
and 
$$\kappa(z_{0},r)=\frac{1}{2\pi r} \, \int\limits_{\gamma(z_{0}, r)} | \mathcal{K}(z)|^2 \, |dz|=\alpha, \quad z_{0}=0.$$
Obviously,  $\kappa(z_{0},r)$ satisfies condition of Corollary \ref{cor1}.

On the other hand, we have $$M_{f}(z_{0},R)=\max\limits_{|z-z_{0}|=R}|f(z)-f(z_{0})|=R^{\frac{1}{\alpha}}$$ and 
$$m_{f}(z_{0},1)=\min\limits_{|z-z_{0}|=1}|f(z)-f(z_{0})|=1.$$
It follows that $$\lim\limits_{R\rightarrow \infty}\frac{M_{f}(z_{0},R)}{R^{\frac{1}{\alpha}}}=1.$$
\end{example}

\begin{corollary}\label{cor2} If for some $\alpha>0$ the condition $|\mathcal{K}_{z_0} (z)|\leqslant\alpha$ holds for a.a. $z\in\{z\in\mathbb C\colon |z-z_0|\geqslant 1\}$, then 
$$
\liminf\limits_{R\to \infty}\frac{M_{f}(z_0,R)}{R^{\beta}}\geqslant m_{f}(z_0,1)>0,
$$
where $\beta=1/\alpha^{2}$ and $m_{f}(z_0,1)=\min\limits_{|z-z_0|=1}|f(z)-f(z_0)|.$
\end{corollary}

Later on we denote
\begin{equation*}
e_{1}=e,\,e_{2}=e^{e},\,...,\,e_{k+1}=e^{e_{k}}
\end{equation*}
and
\begin{equation*}
\ln_{1}t=\ln t,\, \ln_{2}t=\ln\ln t,\,...,\,\ln_{k+1}t=\ln\ln_{k}t,
\end{equation*}
where $k\geqslant1$ are integer.

\begin{theorem}\label{th2} Let $f\colon  \mathbb C\to \mathbb
C$ be a regular homeomorphic solution of the equation (\ref{eqcomplexJ}) which belongs to  Sobolev class $W^{1, 2}_{\rm
loc}$ and $\alpha>0$. If $$\kappa(z_0,r)\leqslant\alpha\prod\limits_{k=1}^{N}\ln_{k}r$$ for a.a. $r\geqslant e_{N}$, then 
$$
\liminf\limits_{R\to \infty}\frac{M_{f}(z_0,R)}{(\ln_{N}R)^{1/\alpha}}\geqslant m_{f}(z_0,e_{N})>0,
$$
where $m_{f}(z_0,e_{N})=\min\limits_{|z-z_0|=e_{N}}|f(z)-f(z_0)|.$
\end{theorem}
\begin{proof}
Note that 
$$\int\limits_{e_{N}}^{R}\frac{dr}{r\,\kappa(z_0,r)}\geqslant\frac{1}{\alpha}\int\limits_{e_{N}}^{R}\frac{dr}{r\,\prod\limits_{k=1}^{N}\ln_{k}r}
=\frac{1}{\alpha}\int\limits_{1}^{\ln_{N}R}\frac{dt}{t}=\ln(\ln_{N}R)^{1/\alpha}.$$
Hence, 
$$\exp\left(-  \int\limits_{e_{N}}^R \,\frac{dr}{r \,  \kappa(z_0,r)}\right)\leqslant\exp\left(-\ln(\ln_{N}R)^{1/\alpha}\right)=\frac{1}{(\ln_{N}R)^{1/\alpha}}.$$
Thus, by Theorem \ref{th1}, we obtain
$$
\liminf\limits_{R\to \infty}\frac{M_{f}(z_0,R)}{(\ln_{N}R)^{1/\alpha}}
\geqslant\liminf\limits_{R\to \infty}\frac{M_{f}(z_0,R)}{\exp\left(\int\limits_{e_{N}}^R \,\frac{dr}{r \,  \kappa(z_0,r)}\right)}\geqslant m_{f}(z_0,e_{N}).
$$
\end{proof}

\begin{corollary}\label{cor4} If for some $\alpha>0$ $$|\mathcal{K}_{z_0} (z)|\leqslant\alpha\left(\prod\limits_{k=1}^{N}\ln_{k}|z-z_0|\right)^{1/2}$$ for a.a. $z\in\{z\in\mathbb C\colon |z-z_0|\geqslant e_{N}\}$, then 
$$
\liminf\limits_{R\to \infty}\frac{M_{f}(z_0,R)}{(\ln_{N}R)^{\beta}}\geqslant m_{f}(z_0,e_{N})>0,
$$
where $\beta=1/\alpha^{2}$ and $m_{f}(z_0,e_{N})=\min\limits_{|z-z_0|=e_{N}}|f(z)-f(z_0)|.$
\end{corollary}

Letting $n=1$ in Theorem~\ref{th2}  gives

\begin{corollary}\label{cor3} If for some $\alpha>0$ the condition $$\kappa(z_0,r)\leqslant\alpha\ln r$$ holds for a.a. $r\geqslant e$,  then 
$$
\liminf\limits_{R\to \infty}\frac{M_{f}(z_0,R)}{(\ln R)^{1/\alpha}}\geqslant m_{f}(z_0,e)>0,
$$
where $m_{f}(z_0,e)=\min\limits_{|z-z_0|=e}|f(z)-f(z_0)|.$
\end{corollary}

\begin{corollary}\label{cor5} If for some $\alpha>0$ the condition $$|\mathcal{K}_{z_0} (z)|\leqslant\alpha(\ln|z-z_0|)^{1/2}$$ holds for a.a. $z\in\{z\in\mathbb C\colon |z-z_0|\geqslant e\}$, then 
$$
\liminf\limits_{R\to \infty}\frac{M_{f}(z_0,R)}{(\ln R)^{\beta}}\geqslant m_{f}(z_0,e)>0,
$$
where $\beta=1/\alpha^{2}$ and $m_{f}(z_0,e)=\min\limits_{|z-z_0|=e}|f(z)-f(z_0)|.$
\end{corollary}

\begin{example}
Assume that $\alpha>0$. Consider the equation
\begin{equation}\label{1111**}
f_{\overline{z}}\,-\frac{z}{\overline{z}}\,f_{z}=\mathcal{K}(z)\, |J_f(z)|^{\frac{1}{2}},
\end{equation}
where $$\mathcal{K}(z)=\begin{cases}-\alpha^{\frac{1}{2}}
(\ln|z|)^{\frac{1}{2}}(\ln\ln|z|)^{\frac{1}{2}}\frac{z}{\overline{z}},\,\,|z|\geqslant e^{e},\\
-\frac{z}{\overline{z}},\,\,|z|< e^{e},\end{cases}$$
in the complex plane $\mathbb{C}.$ In the polar coordinates system, this equation takes the form
$$f_{\theta}=\sigma|J_{f}|^{\frac{1}{2}},$$
where $$\sigma(re^{i\theta})=\begin{cases}i\alpha^{\frac{1}{2}}
(\ln r)^{\frac{1}{2}}(\ln\ln r)^{\frac{1}{2}}re^{i\theta},\,\,r\geqslant e^{e},\\
ire^{i\theta},\,\,0\leqslant r< e^{e}.\end{cases}$$
It's easy to check that the mapping  $$f=\begin{cases}(\ln\ln|z|)^{\frac{1}{\alpha}}\frac{z}{|z|},\,\,|z|\geqslant e^{e},\\
e^{-e}z,\,\,|z|< e^{e}\end{cases}$$
is regular homeomorphism and belongs to Sobolev class $W_{\rm loc}^{1,2}(\mathbb{C})$.
We show that $f$ is a solution of the equation \eqref{1111**}. We write this mapping in the polar coordinates $$f(re^{i\theta})=\begin{cases}(\ln\ln  r)^{\frac{1}{\alpha}}e^{i\theta},\,\,r\geqslant e^{e},\\
e^{-e}re^{i\theta},\,\,0\leqslant r< e^{e}.\end{cases}$$
The partial derivatives of $f$ by $r$ and $\theta$ are
\begin{equation*}
f_{r}=\begin{cases}\alpha^{-1}(\ln\ln r)^{\frac{1-\alpha}{\alpha}}
(\ln r)^{-1}r^{-1}e^{i\theta},\,\,r\geqslant e^{e},\\
e^{-e}e^{i\theta},\,\,0\leqslant r< e^{e}\end{cases}
\end{equation*}
and
$$ f_{\theta}=\begin{cases}i(\ln\ln r)^{\frac{1}{\alpha}}
e^{i\theta},\,\,r\geqslant e^{e},\\
ie^{-e}re^{i\theta},\,\,0\leqslant r< e^{e}.\end{cases}$$
Consequently by \eqref{Jacob} we find $$J_f(re^{i\theta})= \begin{cases}\alpha^{-1}(\ln\ln r)^{\frac{2-\alpha}{\alpha}}
(\ln r)^{-1}r^{-2},\,\,r\geqslant e^{e},\\
e^{-2e},\,\,0\leqslant r< e^{e}\end{cases}$$
and
$$\sigma(re^{i\theta})=\frac{f_{\theta}}{J_f^{\frac{1}{2}}}=\begin{cases}i\alpha^{\frac{1}{2}}
(\ln\ln r)^{\frac{1}{2}}(\ln r)^{\frac{1}{2}}re^{i\theta},\,\,r\geqslant e^{e},\\
ire^{i\theta},\,\,0\leqslant r< e^{e}.\end{cases}$$
Hence,
$$\sigma=\begin{cases}i\alpha^{\frac{1}{2}}
(\ln\ln|z|)^{\frac{1}{2}}(\ln|z|)^{\frac{1}{2}}z,\,\,|z|\geqslant e^{e},\\
iz,\,\,|z|< e^{e}.\end{cases}$$
Next due to the relation \eqref{sigmaK} 
$$\mathcal{K}(z)=-\frac{\sigma(z)}{i\,\overline{z}}=\begin{cases}-\alpha^{\frac{1}{2}}
(\ln\ln|z|)^{\frac{1}{2}}(\ln|z|)^{\frac{1}{2}}\frac{z}{\overline{z}},\,\,|z|\geqslant e^{e},\\
-\frac{z}{\overline{z}},\,\,|z|< e^{e}\end{cases}$$
and for $z_{0}=0$ we have 
$$\kappa(z_{0},r)=\frac{1}{2\pi r} \, \int\limits_{\gamma(z_{0}, r)} | \mathcal{K}(z)|^2 \, |dz|=
\begin{cases}\alpha \ln r \cdot\ln\ln r,\,\,r\geqslant e^{e},\\
1,\,\,0< r< e^{e}.\end{cases}$$
Note that $\kappa(z_{0},r)$ does not satisfy condition of Corollary \ref{cor3}, because
$$\lim\limits_{r\to\infty}\frac{\kappa(z_{0},r)}{\ln r}=\infty.$$

On the other hand, we have $M_{f}(z_{0},R)=(\ln\ln R)^{\frac{1}{\alpha}}$. 
It follows that $$\lim\limits_{R\rightarrow \infty}\frac{M_{f}(z_{0},R)}{(\ln R)^{\frac{1}{\alpha}}}=0.$$
\end{example}

\section{Non-existence theorems}

In this section we find sufficient conditions under which the equation \eqref{eqcomplexJ} has no homeomorphic regular solutions belonging to the Sobolev class $W^{1, 2}_{\rm loc}$.

\begin{theorem}\label{th3} Let $\mathcal{K}_{z_0}  \colon \mathbb{C} \to\mathbb{C}$ be a measurable function. Then there are no regular homeomorphic solutions $f\colon  \mathbb C\to \mathbb
C$ of the equation (\ref{eqcomplexJ}) from  Sobolev class $W^{1, 2}_{\rm
loc}$ with asymptotic condition
\begin{equation}\label{eq33}
\liminf\limits_{R\to \infty}M_{f}(z_0,R)\exp\left(-  \int\limits_{r_0}^R \,\frac{dr}{r \,  \kappa(z_0,r)}\right)=0
\end{equation}
for some $r_0>0$, where $M_{f}(z_0,R)=\max\limits_{|z-z_0|=R}|f(z)-f(z_0)|.$
\end{theorem}
\begin{proof}
Suppose that there is a regular homeomorphic solution\break $f\colon  \mathbb C\to \mathbb
C$  of the equation (\ref{eqcomplexJ}) from  Sobolev class $W^{1, 2}_{\rm
loc}$ for which the condition \eqref{eq33} is satisfied. Then, by the Theorem \ref{th1}, we get
$$m_{f}(z_0,r_0)=\min\limits_{|z-z_0|=r_0}|f(z)-f(z_0)|=0.$$ 
This contradicts the fact that the mapping $f$ is homeomorphic.
\end{proof}

\begin{corollary} Let $\mathcal{K}_{z_0}  \colon \mathbb{C} \to\mathbb{C}$ be a measurable function  and $\alpha>0$. If $\kappa(z_0,r)\leqslant\alpha$ for a.a. $r\geqslant 1$, then there are no regular homeomorphic solutions $f\colon  \mathbb C\to \mathbb
C$ of the equation (\ref{eqcomplexJ}) from  Sobolev class $W^{1, 2}_{\rm
loc}$ with asymptotic condition
\begin{equation*}
\liminf\limits_{R\to \infty}\frac{M_{f}(z_0,R)}{R^{1/\alpha}}=0,
\end{equation*}
where $M_{f}(z_0,R)=\max\limits_{|z-z_0|=R}|f(z)-f(z_0)|.$
\end{corollary}

\begin{corollary} Let $\mathcal{K}_{z_0}  \colon \mathbb{C} \to\mathbb{C}$ be a measurable function and $\alpha>0$. If $|\mathcal{K}_{z_0} (z)|\leqslant\alpha$ for a.a. $z\in\{z\in\mathbb C\colon |z-z_0|\geqslant 1\}$, then there are no regular homeomorphic solutions $f\colon  \mathbb C\to \mathbb
C$ of the equation (\ref{eqcomplexJ}) from  Sobolev class $W^{1, 2}_{\rm
loc}$ with asymptotic condition
\begin{equation*}
\liminf\limits_{R\to \infty}\frac{M_{f}(z_0,R)}{R^{\beta}}=0,
\end{equation*}
where  $\beta=1/\alpha^{2}$.
\end{corollary}

\begin{corollary} Let $\mathcal{K}_{z_0}  \colon \mathbb{C} \to\mathbb{C}$ be a measurable function and $\alpha>0$. If $$\kappa(z_0,r)\leqslant\alpha\prod\limits_{k=1}^{N}\ln_{k}r$$ for a.a. $r\geqslant e_{N}$, then there are no regular homeomorphic solutions\break $f\colon  \mathbb C\to \mathbb
C$ of the equation (\ref{eqcomplexJ}) from  Sobolev class $W^{1, 2}_{\rm
loc}$ with asymptotic condition
\begin{equation*}
\liminf\limits_{R\to \infty}\frac{M_{f}(z_0,R)}{(\ln_{N}R)^{1/\alpha}}=0,
\end{equation*}
where $M_{f}(z_0,R)=\max\limits_{|z-z_0|=R}|f(z)-f(z_0)|.$
\end{corollary}

\begin{corollary} Let $\mathcal{K}_{z_0}  \colon \mathbb{C} \to\mathbb{C}$ be a measurable function and $\alpha>0$. If $$|\mathcal{K}_{z_0} (z)|\leqslant\alpha\left(\prod\limits_{k=1}^{N}\ln_{k}|z-z_0|\right)^{1/2}$$ for a.a. $z\in\{z\in\mathbb C\colon |z-z_0|\geqslant e_{N}\}$, then there are no regular homeomor\-phic solutions $f\colon  \mathbb C\to \mathbb
C$ of the equation (\ref{eqcomplexJ}) from  Sobolev class $W^{1, 2}_{\rm
loc}$ with asymptotic condition
\begin{equation*}
\liminf\limits_{R\to \infty}\frac{M_{f}(z_0,R)}{(\ln_{N}R)^{\beta}}=0,
\end{equation*}
where  $\beta=1/\alpha^{2}$.
\end{corollary}

\begin{corollary} Let $\mathcal{K}_{z_0}  \colon \mathbb{C} \to\mathbb{C}$ be a measurable function  and $\alpha>0$. If $$\kappa(z_0,r)\leqslant\alpha\ln r$$ for a.a. $r\geqslant e$, then there are no regular homeomorphic solutions $f\colon  \mathbb C\to \mathbb
C$ of the equation (\ref{eqcomplexJ}) from  Sobolev class $W^{1, 2}_{\rm
loc}$ with asymptotic condition
\begin{equation*}
\liminf\limits_{R\to \infty}\frac{M_{f}(z_0,R)}{(\ln R)^{1/\alpha}}=0,
\end{equation*}
where $M_{f}(z_0,R)=\max\limits_{|z-z_0|=R}|f(z)-f(z_0)|.$
\end{corollary}

\begin{corollary} Let $\mathcal{K}_{z_0}  \colon \mathbb{C} \to\mathbb{C}$ be a measurable function and $\alpha>0$. If $$|\mathcal{K}_{z_0} (z)|\leqslant\alpha (\ln|z-z_0|)^{1/2}$$ for a.a. $z\in\{z\in\mathbb C\colon |z-z_0|\geqslant e\}$, then  there are no regular homeomorphic solutions $f\colon  \mathbb C\to \mathbb
C$ of the equation (\ref{eqcomplexJ}) from  Sobolev class $W^{1, 2}_{\rm
loc}$ with asymptotic condition
\begin{equation*}
\liminf\limits_{R\to \infty}\frac{M_{f}(z_0,R)}{(\ln R)^{\beta}}=0,
\end{equation*}
where  $\beta=1/\alpha^{2}$.
\end{corollary}

\bigskip

\bigskip

\textbf{Acknowledgement}

This work was supported by a grant from the Simons Foundation (1290607, Ruslan Salimov, Mariia Stefanchuk).

\bigskip

CONTACT INFORMATION

\medskip

I.~Petkov \\Admiral Makarov National University of Shipbuilding,\\Ukraine, Mykolaiv, 9 Heroes of Ukraine Avenue, 54007,\\igorpetkov83@gmail.com

\medskip

R.~Salimov\\Institute of Mathematics of NAS of Ukraine,\\3 Tereschenkivska St., Kiev-4, 01024, Ukraine,\\ruslan.salimov1@gmail.com

\medskip
M.~Stefanchuk\\Institute of Mathematics of NAS of Ukraine,\\3 Tereschenkivska St., Kiev-4, 01024, Ukraine,\\stefanmv43@gmail.com


\begin{thebibliography}{100}

\bibitem{AC} C.~ Andreian Cazacu, "Influence of the orientation of the characteristic ellipses on the properties of the quasiconformal mappings," {\em Publishing House of the Academy of the Socialist Republic of Romania, Bucharest,} 65--85 (1971).
    
\bibitem{RSY}     V.~Ryazanov, U.~Srebro,  and E.~Yakubov, "On ring solutions of Beltrami equations,"
{\em  J. Anal. Math.}, \textbf{96}, 117--150 (2005).

\bibitem{Golberg1} A.~Golberg, "Directional dilatations in space," {\em Complex Var. Elliptic Equ.}, \textbf{55}(1--3), 13--29 (2010).

\bibitem{Golberg2}  A.~Golberg, "Extremal bounds of Teichm\"{u}ller-Wittich-Belinski\u\i\ type for planar quasiregular mappings,"
{\em Fields Inst. Commun., 81 Springer, New York,}, 173--199  (2018).

\bibitem{GRSY12}
V.~Gutlyanski\u\i, V.~Ryazanov, U.~Srebro, and E.~Yakubov, {\em The Beltrami equation. A geometric approach.} Developments in Mathema\-tics, \textbf{26}, Springer, New York, 2012.

\bibitem{MRSY09}
O.~Martio, V.~Ryazanov, U.~Srebro, and E.~Yakubov, {\em Moduli in
Modern Mapping Theory}. Springer Monographs in Mathematics., Springer, New York, 2009.

\bibitem{GRSY10}
V.~Gutlyanski\u\i, V.~Ryazanov, U.~Srebro, and E.~Yakubov, "On recent advances in the Beltrami equations," {\em J. Math. Sci.,}, \textbf{175}(4), 413--449 (2011).

\bibitem{SY05} U.~Srebro and E.~Yakubov, "Beltrami equation," {\em Handbook of complex analysis: geometric function theory, Elsevier Sci. B. V., Amsterdam,} \textbf{2}, 555--597 (2005).
    
\bibitem{GSS}
A.~Golberg, R.~Salimov, and M.~Stefanchuk, "Asymptotic dilation of re\-gular homeomorphisms," {\em Complex Anal. Oper. Theory}, \textbf{13}(6), 2813--2827 (2019).

\bibitem{SalSt1}
R.~R.~Salimov and M.~V.~Stefanchuk, "On the local properties of solutions of the nonlinear Beltrami equation," {\em J. Math. Sci.}, \textbf{248}, 203--216  (2020).

\bibitem{GSal} A.~Golberg and R.~Salimov, "Nonlinear Beltrami equation," {\em Complex Var. Elliptic Equ.}, \textbf{65}(1), 6--21 (2020).

\bibitem{SalSt2} R.~R.~Salimov and M.~V.~Stefanchuk, "Logarithmic asymptotics of the nonlinear Cauchy-Riemann-Beltrami equation," {\em Ukr. Math. J.,} \textbf{73}, 463--478 (2021).
    
\bibitem{SalSt4} R.~Salimov and M.~Stefanchuk, "Finite Lipschitzness of regular solutions to nonlinear Beltrami equation", {\em Complex Var. Elliptic Equ.} (2023), DOI: https://doi.org/10.1080/17476933.2023.2166498
    
\bibitem{BGMR13}
B.~Bojarski, V.~Gutlyanski\u\i, O.~Martio, and V.~Ryazanov, {\em Infinitesimal geometry of quasiconformal and bi-Lipschitz mappings in the plane}.
EMS Tracts in Mathematics, \textbf{19}, European Mathematical Society (EMS), Z\"urich, 2013.

\bibitem{Sev11} E.~A.~Sevost'yanov, "On quasilinear Beltrami-type equations with dege\-neration," {\em Math. Notes,} \textbf{90}(3--4), 431--438 (2011).
    
\bibitem{AIM09} K.~Astala, T.~Iwaniec, and G.~Martin, {\em Elliptic partial differential equations and quasiconformal mappings in the plane}.
Princeton Mathematical Series, \textbf{48}, Princeton University Press, Princeton, NJ, 2009.

\bibitem{LSh57} M.~A.~Lavrent'ev and B.~V.~\v{S}abat, "Geometrical properties of solutions of non-linear systems of partial differential equations," {\em Dokl. Akad. Nauk SSSR (N.S.),} \textbf{112}, 810--811 (1957) (in Russian).

\bibitem{Lavr47} M.~A.~Lavrent'ev, "A general problem of the theory of quasi-conformal representation of plane regions,"  {\em Mat. Sbornik N.S.,} \textbf{21}(63),  285--320 (1947) (in Russian).

\bibitem{Lavr62} M.~A.~Lavrent'ev, {\em The variational method in boundary-value problems for systems of equations of elliptic type}. Izdat. Akad. Nauk SSSR, Moscow, 1962 (in Russian).

\bibitem{SalSt3}  R.~R. Salimov and M.~V.  Stefanchuk, "Nonlinear Beltrami equation and asymptotics of its solution," {\em J. Math. Sci.,} {\bf 264}(4), 441--454 (2022).

\bibitem{KlSalStef222} B. Klishchuk, R. Salimov, and M. Stefanchuk, "Schwarz lemma type estimates for solutions to nonlinear Beltrami equation," {\it Analysis, Applications, and Computations. Trends in Mathematics,}  295--305 (2023).
    
\bibitem{PetSalStef22} I. Petkov,  R. Salimov, and M. Stefanchuk,  "Nonlinear Beltrami equation: lower estimates of Schwarz Lemma’s type," {\em Canadian Mathematical Bulletin} (2023), DOI: https://doi.org/10.4153/S0008439523000942
    
\bibitem{SalStef22} R.~R. Salimov and M.~V.  Stefanchuk,  "Functional asymptotics of solutions of the nonlinear Cauchy–Riemann–Beltrami system," {\em J. Math. Sci.,} \textbf{277}, 311--328 (2023).
    
\bibitem{SalStef}  R.~R. Salimov and M.~V. Stefanchuk, "On one extremal problem for nonlinear Cauchy-Riemann-Beltrami systems," {\em Pratsi IPMM NAN Ukrainy,} \textbf{34}, 109--115 (2020) (in Ukrainian).
    
\bibitem{Stef} M.~V. Stefanchuk,  "On extremal problems of exponential type for solutions of the nonlinear Beltrami equation," {\em Pratsi IPMM NAN Ukrainy,} \textbf{36}(1), 36--43 (2022) (in Ukrainian).

\bibitem{GK16} C.-Y.~Guo and M.~Kar, "Quantitative uniqueness estimates for $p$-Laplace type equations in the plane," {\em Nonlinear Anal.,} \textbf{143}, 19--44, (2016).

\bibitem{Shab57} B.~V.~\v{S}abat, "Geometric interpretation of the concept of ellipticity," {\em Uspehi Mat. Nauk,} \textbf{12}(6)(78), 181--188 (1957) (in Russian).

\bibitem{Shab61} B.~V.~\v{S}abat, "On the notion of derivative system according to M.~A.~Lavrent'ev," {\em Soviet Math. Dokl.,} \textbf{2}, 202--205 (1961).

\bibitem{Kuh10} R.~K\"uhnau, "Minimal surfaces and quasiconformal mappings in the mean,"
{\em Trans. of Institute of Mathematics, National Academy of Sciences of Ukraine,} \textbf{7}(2), 104--131 (2010).

\bibitem{KK83} S.~L.~Kruschkal and R.~K\"uhnau,  {\em Quasikonforme Abbildungen~---neue Me\-thoden und Anwendungen.} (in German). {\em Quasiconformal mappings~---new methods and applications.} With English, French and Russian summaries. Teubner-Texte zur Mathematik (Teubner Texts in Mathematics), \textbf{54}, BSB B. G. Teubner Verlagsgesellschaft, Leipzig, 1983.

\bibitem{Adam09} T.~Adamowicz, "On $p$-harmonic mappings in the plane," {\em Nonlinear Anal.,} \textbf{71}(1--2), 502--511 (2009).

\bibitem{Aron88} G.~Aronsson, "On certain $p$-harmonic functions in the plane," {\em Manuscripta Math.,} \textbf{61}(1), 79--101 (1988).

\bibitem{Rom08}
A.~S.~Romanov, "Capacity relations in a planar quadrilateral," {\em Sib. Math. J.,} \textbf{49}(4), 709--717 (2008).

\bibitem{BI87}
B.~Bojarski and T.~Iwaniec, "$p$-harmonic equation and quasiregular mappings,"
{\em Partial differential equations (Warsaw, 1984),  Banach Center Publ.,} \textbf{19}, PWN, Warsaw, 25--38 (1987).

\bibitem{ACFJK17}
K.~Astala, A.~Clop, D.~Faraco, J.~J\"a\"askel\"ainen, and A.~Koski,  "Nonlinear Beltrami operators, Schauder estimates and bounds for the Jacobian,"
{\em Ann. Inst. H. Poincar\'e Anal. Non Lin\'eaire,} \textbf{34}(6), 1543--1559 (2017).

\bibitem{CGPSS18}
M.~Carozza, F.~Giannetti, A.~Passarelli di Napoli, C.~Sbordone, and R.~Schiattarella,  "Bi-Sobolev mappings and $K_p$-distortions in the plane,"
{\em J. Math. Anal. Appl.,} \textbf{457}(2), 1232--1246 (2018).

\bibitem{LV73}
O.~Lehto and K.~I.~Virtanen, {\em Quasiconformal mappings in the plane}. Second edition. Translated from the German by K. W. Lucas. Die Grundlehren der mathematischen Wissenschaften, \textbf{126}, Springer-Verlag, New York-Heidelberg, 1973.
    
\bibitem{MM} J. Maly and O. Martio, "Lusin's condition $N$ and mappings of the class $W_{loc}^{1,n}$," {\em J. Reine Angew. Math.,}  \textbf{458}, 19--36 (1995).


\bibitem{Saks64}
S.~Saks,  {\em Theory of the Integral}. Dover, New York, 1964.

\end{thebibliography}
\end{document}